\documentclass[11 pt,dvipsnames]{article}
\usepackage{harvard}
\usepackage{graphicx}
\usepackage{camnum}
\usepackage{amsmath}
\usepackage{amsfonts}
\usepackage{epstopdf}
\usepackage{xcolor}

\numberwithin{figure}{section}
\numberwithin{equation}{section}
\numberwithin{table}{section}
\numberwithin{theorem}{section}

\allowdisplaybreaks

\newcommand{\ee}{{\mathrm e}}
\newcommand{\ii}{{\mathrm i}}

\newcommand{\LT}{\mathrm{LT}}
\newcommand{\PLT}{\mathrm{PLT}}
\newcommand{\St}{\mathrm{S}}

\date{}

\begin{document}
\thispagestyle{empty}

\title{An elementary approach to splittings of unbounded operators}

\author{Arieh Iserles \& Karolina Kropielnicka}

\maketitle


\begin{abstract}
  Using elementary means, we derive the three most popular splittings of $\ee^{t(A+B)}$ and their error bounds in the case when $A$ and $B$ are (possibly unbounded) operators  in a Hilbert space, generating strongly continuous semigroups, $\ee^{tA}$, $\ee^{tB}$ and $\ee^{t(A+B)}$.  The error of these splittings is bounded in terms of the norm of the commutators $[A,B]$, $[A,[A,B]]$ and $[B,[A,B]]$.
\end{abstract}

\section{Introduction}

A popular approach in numerical analysis is to approximate a problem that we wish to compute by a combination of sub-problems whose solution is more accessible  \cite{blanes24smd,mclachlan02sm}. This leads to powerful computational algorithms but mathematical analysis of this procedure is often highly nontrivial.

Exponential splittings arise when we attempt to solve a complicated linear differential system whose `sub-problems' are much easier to compute. They  play a prominent role in a wide range of computational methods for evolutionary differential equations. Yet,  the analysis of even the simplest splitting presents a substantial problem in the context of partial differential equations. 

The starting point is the abstract linear equation,
\begin{equation}
  \label{eq:1.1}
  \frac{\D u}{\D t}=(A+B)u,\quad t\geq0,\qquad u(0)=u_0,
\end{equation}
where, for simplicity's sake, we assume zero boundary conditions. Here $A$ and $B$ are linear differential operators, which might be unbounded. We stipulate that the solution of \R{eq:1.1} is {\em well posed:\/} for every $t\geq0$ there exists $c_t$ such that $\|u(t)\|\leq c_t \|u_0\|$ -- in other words, small perturbations in the initial condition lead to bounded changes in the solution for all $t\geq0$.

We denote the exact solution of \R{eq:1.1}  by $u(t)=\ee^{t(A+B)}u_0$: if $A$ and $B$ are matrices, the exponential can be (at least in principle) computed by a Taylor expansion but this is no longer the case once $A$ and $B$ are unbounded linear operators: we cannot  regard them as  `very large matrices'. However, in many instances computing both $\ee^{tA}u_0$ (the solution of $\D u/\D t=Au$, $u(0)=u_0$)   and $\ee^{tB}u_0$ is easy (whether exactly or to a high degree of precision), hence  $\ee^{t(A+B)}$ (for small $t>0$ -- the numerical procedure is typically accomplished by time stepping and $t$ is the step size) might be approximated by the {\em Lie--Trotter splitting\/}
\begin{equation}
  \label{eq:1.2}
  Y_{\LT}(A,B;t):=\ee^{tA}\ee^{tB}.
\end{equation}
If $A$ and $B$ commute, that is $[A,B]=AB-BA=O$, it is true that $Y_{\LT}(A,B;t)=\ee^{t(A+B)}$, otherwise $Y_{\LT}(A,B;t)\neq\ee^{t(A+B)}$ but $Y_{\LT}$ might still be a good approximation to the exact exponential, thus to the solution of \R{eq:1.1}. 

As long as \R{eq:1.1} is an ordinary differential equation, everything reduces to matrix analysis and it is trivial to expand the error into Taylor series, proving that $Y_{\LT}(A,B;t)-\ee^{t(A+B)}=\O{t^2}$: Lie--Trotter is hence a first-order approximation. Not so, however, once either $A$ or $B$ are unbounded operators. 

A helpful paradigm is the {\em linear Schr\"odinger equation\/}
\begin{equation}
  \label{eq:1.3}
  \ii \frac{\partial u}{\partial t}=-\Delta u+V(x)u,\qquad t\geq0, \quad x\in\mathbb{R}^d,
\end{equation}
with an initial condition given in a suitable Hilbert space ${\mathcal X}\subseteq\CC{L}_2(\mathbb{R}^d)$. Thus, $A=\ii\Delta$ and $B=-\ii V$ and we note that the differential operator $A$, which corresponds to the {\em kinetic component\/} of a quantum system, is unbounded. In some cases the {\em interaction potential\/} $V$ might be also unbounded. 

The operators $A+B$, $A$ and $B$ are generators of strongly continuous semigroups $\ee^{t(A+B)}$, $\ee^{tA}$ and $\ee^{t B}$ respectively\footnote{Strong continuity of $\ee^{t(A+B)}$ implies that $\|\ee^{t(A+B)}\|\leq c\ee^{\omega t}$ for some $c,\omega>0$ and all $t\geq0$, meaning that the operator $A+B$ is well posed for the problem \R{eq:1.1} \cite{pazy83slo}.}  hence {\em we are not allowed to use a Taylor expansion in the presence of unbounded operators,\/} cf.\  \cite{pazy83slo}. Therefore the standard, trivial proof that $Y_{\LT}$ is a first-order splitting -- that is, that $Y_{\LT}(A,B;t)u_0=\ee^{t(A+B)}u_0+\O{t^2}$ for every $u_0\in\mathcal{X}$ -- is no longer valid! 

The subject of this paper is detailed analysis of three most familiar splittings: the Lie--Trotter splitting \R{eq:1.2}, the {\em palindromic Lie--Trotter splitting\/}
\begin{equation}
  \label{eq:1.4}
  Y_{\PLT}(A,B;t):=\frac12 (\ee^{tA}\ee^{tB}+\ee^{tB}\ee^{tA})=\frac12 [Y_{\LT}(A,B;t)+Y_{\LT}(B,A;t)]
\end{equation}
(in this context `palindromic' means that it is the same whether `read' left-to-right or right-to-left) and the {\em Strang splitting\/} 
\begin{equation}
  \label{eq:1.5}
  Y_{\St}(A,B;t):=\ee^{\frac12 tA}\ee^{tB}\ee^{\frac12 tA}=Y_{\LT}(A,B;\tfrac12 t) Y_{\LT}(B,A;\tfrac12 t)
\end{equation}
\cite{mclachlan02sm}. It is tempting to consider palindromic Lie--Trotter as an arithmetic mean of Lie--Trotter and its `reverse' $Y_{\LT}(B,A;t)=\ee^{tB}\ee^{tA}$ and Strang's splitting as their geometric mean, but we make no further use of this observation.

An interesting aside: the Strang splitting is {\em symmetric,\/} $Y_{\St}^{-1}(A,B;t)=Y_{\St}(A,B,-t)$. This is an easy exercise, an outcome of palindromy and in finite-dimensional or bounded setting implies that there exists an operator $\Omega$ such that $Y_{\St}(A,B;t)=\ee^{\Omega(t)}$ and $\Omega(-t)=-\Omega(t)$ \cite{hairer06gni,mclachlan02sm}.\footnote{The standard means of constructing $\Omega$ is the symmetric Baker--Campbell--Hausdorff formula \cite{iserles00lgm}.} An important consequence is that its order can be increased using the Yoshida device \cite{yoshida90cho}. This, however is not the case with palindromic Lie--Trotter because it is easy to verify that it is not symmetric, even though it is palindromic.

An analysis of the Strang splitting for unbounded operators has been presented in the path-breaking paper of \citeasnoun{jahnke00ebe}. They assumed $B$ to be bounded, $\|\ee^{tA}\|,\|\ee^{tB}\|,\|\ee^{t(A+B)}\|\leq1$ for $t\geq0$ and the existence of $\alpha,\beta>0$ such that
\begin{displaymath}
  \|[A,B]u\|\leq c_1\|(-A)^\alpha u\|,\qquad \|[A,[A,B]]u\|\leq c_2\|(-A)^\beta u\|,\qquad u\in\mathcal{X},
\end{displaymath}
for some $c_1,c_2>0$ (all this is true for the linear Schr\"odinger equation \R{eq:1.3}), proved that \R{eq:1.5} is of order two in this setting,
\begin{displaymath}
  Y_{\St}(A,B;t)u=\ee^{t(A+B)}u+\O{t^3},\qquad u\in\mathcal{X},
\end{displaymath}
and presented an error bound. Cf.\ \cite{hansen09esu} for the case when time-independent $A$ and $B$ are both unbounded. An important elaboration upon the model \R{eq:1.1} is when $B$ is allowed to depend on $t$, and this can be accommodated into the formalism of exponential splittings. We refer to \cite{jahnke00ebe} and \cite{delvalle03fst} for details.

We assume for the time being just that the semigroups $\ee^{tA},\ee^{tB}$ and $\ee^{t(A+B)}$ are well posed in ${\mathcal X}$: this is sufficient in our quest to express in Section~2 the errors of different splittings in an accessible form. The next step (and the theme of Section~3) is to use these forms to derive explicit upper bounds on the error and this requires stricter assumptions on functions $u\in\mathcal{X}$, e.g.\ that $\|[A,B]u\|$, $\|[[A,B],A]u\|$ and $\|[[A,B],B]u\|$ are all bounded.

In Section~2 we derive explicit expressions for the error of Lie--Trotter, palindromic Lie--Trotter and Strang splittings in terms of multiple integrals, thereby determining their order. Firstly we derive the error
\begin{displaymath}
  E_{\LT}(A,B;t)=Y_{\LT}(A,B;t)-\ee^{t(A+B)}
\end{displaymath}
from basic principles. Secondly, we represent
\begin{displaymath}
  E_{\PLT}(A,B;t)=Y_{\PLT}(A,B;t)-\ee^{t(A+B)}
\end{displaymath}
in terms of $E_{\LT}(A,B;t)$ and $E_{\LT}(B,A;t)$ and, finally, we represent
\begin{displaymath}
  E_{\St}(A,B;t)=Y_{\St}(A,B;t)-\ee^{t(A+B)}
\end{displaymath}
using nested integrals.

Once $A$ and $B$ are matrices, trivial Taylor expansion affirms that
\begin{Eqnarray}
  \label{eq:1.6}
  E_{\LT}(A,B;t)&=&\frac12[A,B]t^2+\O{t^3},\\
  \label{eq:1.7}
  E_{\PLT}(A,B;t)&=&\frac{1}{12}[A-B,[A,B]]t^3+\O{t^4},\\
  \label{eq:1.8}
  E_{\St}(A,B;t)&=&-\frac{1}{24}[A+2B,[A,B]]t^3+\O{t^4}.
\end{Eqnarray}

The rules of the game in this paper are simple: {\em we are not allowed to expand into Taylor series and, once we seek error bounds, must be careful with the presence of commutators containing unbounded operators.\/} Because of possible unboundedness, such commutators might make sense in upper bounds only for a subset of initial values in $\mathcal{X}$.

While the  expressions from Section~2 are sufficient to obtain order, they are perhaps too opaque as realistic tools to estimate the error. In Section~3 we use logarithmic norms \cite{soderlind2006tln} to produce tight upper bounds on the error committed by our three splittings.

Note that both Lie--Trotter and Strang splitting have their reversed counterparts (i.e., with $A$ and $B$ swapped), $\ee^{tB}\ee^{tA}$ and $\ee^{\frac12 tB}\ee^{tA}\ee^{\frac12 tB}$ respectively.  (Palindromic Lie--Trotter is its own reverse.) All we need, though, to extend our results to the `reverse setting' is to swap the role of $A$ and $B$ in our error terms and corresponding bounds.

The main tool in our analysis consists of demonstrating that each of our three splittings obeys a perturbed differential equation, whose solution affords us a handle on the size of the error. The use of a perturbed differential equation satisfied by a splitting has already featured in  \cite{sheng94gee} in a finite-dimensional setting. The framework of the present paper is substantially different.

\section{Error expressions}

\subsection{The Lie--Trotter splitting}

Throughout this paper derivatives of operators are always applied to the elements of the Hilbert space $\mathcal{X}$ but, for the sake of brevity, we consider them as `proper' derivatives. Thus, for example, $\D \ee^{tA}/\D t=A\ee^{tA}$ really means that $(\D \ee^{tA}/\D t )u=A\ee^{tA}u$ for every $u\in\mathcal{X}$ (or for a subset of $\mathcal{X}$).

It is easy to derive the following result from basic principles,

\begin{theorem}
  \label{th:LieTrotter}
  The error of the Lie--Trotter splitting is
  \begin{equation}
    \label{eq:2.1}
    E_{\LT}(A,B;t)=\int_0^t \ee^{(t-\tau)(A+B)} \int_0^\tau \ee^{(\tau-\xi)A}[A,B]\ee^{\xi A}\D\xi \ee^{\tau B}\D\tau.
  \end{equation}
\end{theorem}

\begin{proof}
  By straightforward differentiation,
  \begin{equation}
    \label{eq:2.X}
    \frac{\D Y_{\LT}(A,B;t)}{\D t}-(A+B)Y_{\LT}(A,B;t)=[\ee^{tA},B]\ee^{t B}
  \end{equation}
  and, of course, $Y_{\LT}(A,B;0)=I$. Recall that the {\em Duhamel principle\/} is the familiar variation of constants formula in the context of partial differential equations: the solution of $X'=CX+F(t)$, where $C$ might be a matrix or an operator, is
  \begin{displaymath}
    X(t)=\ee^{tC}X(0)+\int_0^t \ee^{(t-\tau)C} F(\tau)\D\tau.
  \end{displaymath}
  Letting $C=A+B$ and $F(t)=[\ee^{tA},B]\ee^{tB}$ in \R{eq:2.X} implies that $X(t)=Y_{\LT}(A,B;t)$ and we have
  \begin{displaymath}
    E_{\LT}(A,B;t)=Y_{\LT}(A,B;t)-\ee^{t(A+B)}=\int_0^t \ee^{(t-\tau)(A+B)} [\ee^{\tau A},B]\ee^{\tau B}\D\tau.
  \end{displaymath}
  Letting $Z(\tau)=[\ee^{\tau A},B]$, we obtain
  \begin{displaymath}
    \frac{\D Z(\tau)}{\D\tau}-AZ(\tau)=[A,B]\ee^{\tau A},\qquad Z(0)=O,
  \end{displaymath}
  and using the Duhamel principle again
  \begin{displaymath}
    Z(\tau)=\int_0^\tau \ee^{(\tau-\xi)A}[A,B]\ee^{\xi A}\D\xi.
 \end{displaymath}
 The error expression \R{eq:2.1} now follows by substitution into $E_{\LT}$.
\end{proof}

The presence of two nested integrals in \R{eq:2.1} indicates that $E_{\LT}(A,B;t)=\O{t^2}$ and the order of the Lie--Trotter splitting is one. 

It is sometimes, e.g.\ in the next section, useful to rewrite \R{eq:2.1} in a form that separates between the leading $\O{t^2}$ term and a higher-order one. Let 
 \begin{displaymath}
   U(\tau)=\int_0^\tau \ee^{-\xi A}[A,B]\ee^{\xi A}\D\xi,
 \end{displaymath}
 hence \R{eq:2.1} equals
 \begin{displaymath}
   E_{\LT}(A,B;t)=\int_0^t \ee^{(t-\tau)(A+B)} \ee^{\tau A} U(\tau)\ee^{\tau B}\D\tau.
 \end{displaymath}
 Given any  operator $Q$ acting on $\mathcal{X}$, we let $P_Q(A,B;\xi)=\ee^{-\xi A}Q\ee^{\xi A}$. Therefore
 \begin{displaymath}
  P_Q(A,B;0)=Q,\qquad P_Q'(A,B;\xi)=\ee^{-\xi A}[Q,A]\ee^{\xi A}=-\ee^{-\xi A}[A,Q]\ee^{\xi A}
\end{displaymath} 
and
 \begin{equation}
   \label{eq:2.2}
   P_Q(A,B;\xi)=Q-\int_0^\xi \ee^{-\eta A}[A,Q]\ee^{\eta A}\D\eta
 \end{equation}
 -- the first term on the right is $\O{1}$ and the second $\O{\xi}$. Letting $Q=[A,B]$ we have
 \begin{Eqnarray}
   \nonumber
   E_{\LT}(A,B;t)&=&\int_0^t\ee^{(t-\tau)(A+B)}\ee^{\tau A} \int_0^\tau P_{[A,B]}(A,B;\xi)\D\xi \ee^{\tau B}\D\tau\\
   \nonumber
   &=&\int_0^t \ee^{(t-\tau)(A+B)}\ee^{\tau A}\! \left\{ \tau [A,B]\right.\\
   \nonumber
   &&\left.\mbox{}-\int_0^\tau\!\int_0^\xi \ee^{-\eta A}[A,[A,B]]\ee^{\eta A}\!\D\eta\D\xi\right\}\!\ee^{\tau B}\D\tau \\
   \label{eq:2.3}
   &=&\int_0^t \tau \ee^{(t-\tau)(A+B)}\ee^{\tau A}[A,B]\ee^{\tau B}\D\tau\\
   \nonumber
   &&\mbox{}-\int_0^t \ee^{(t-\tau)(A+B)}\ \int_0^\tau (\tau-\eta)\ee^{(\tau-\eta) A}[A,[A,B]]\ee^{\eta A}\D\eta\, \ee^{\tau B}\D\tau.
 \end{Eqnarray}
 Note that the first integral in \R{eq:2.3} is $\O{t^2}$, while the second is $\O{t^3}$.

\subsection{Palindromic Lie--Trotter}

\begin{theorem}
  \label{th:PLT}
  The error of the palindromic Lie--Trotter splitting is
  \begin{Eqnarray}
     \label{eq:2.4}
    E_{\PLT}(A,B;t)&=&\frac12 \int_0^t \tau \ee^{(t-\tau)(A+B)} \left\{ \ee^{\tau A}[A,B]\ee^{\tau B}-\ee^{\tau B}[A,B]\ee^{\tau A}\right\}\!\D\tau \hspace*{30pt}\\
    \nonumber
    &&\mbox{}-\frac12 \int_0^t \ee^{(t-\tau)(A+B)} \int_0^\tau (\tau-\xi)\left\{ \ee^{(\tau-\xi)A}[A,[A,B]]\ee^{\xi A}\ee^{\tau B} \right.\\
    \nonumber
    &&\hspace*{20pt}\left.\mbox{}-\ee^{(\tau-\xi)B}[B,[A,B]]\ee^{\xi B}\ee^{\tau A}\right\}\!\D\xi\D\tau.
  \end{Eqnarray}
  Moreover, since both terms are $\O{t^3}$, the splitting is of order two.
\end{theorem}

\begin{proof}
  It is obvious that
  \begin{Eqnarray*}
    E_{\PLT}(A,B;t)=\frac12 \left\{ E_{\LT}(A,B;t)+E_{\LT}(B,A;t)\right\}
  \end{Eqnarray*}
  and the theorem follows from \R{eq:2.3} by standard algebra, taking the average of $E_{\LT}(A,B;t)$ and $E_{\LT}(B,A;t)$.
\end{proof}

The expression \R{eq:2.4} is inconvenient for the derivation of error bounds -- the subject of the next section -- because the norm of the first integral cannot be bounded by an expression of the form $ct^3$ for some $c>0$. It is thus convenient to rewrite it. Let $Q$ be an operator acting on $\mathcal{X}$ or its subspace and define
\begin{displaymath}
  R_Q(\tau)=\ee^{\tau A}Q\ee^{\tau B}-\ee^{\tau B}Q\ee^{\tau A}
\end{displaymath}
-- note that the term in curly brackets in the first integral in \R{eq:2.4} is $R_{[A,B]}(\tau)$. Since $R_Q(0)=O$ and
\begin{displaymath}
  R_Q'(\tau)-(A+B)R_Q(\tau)=[\ee^{\tau A}Q,B]\ee^{\tau B}+[A,\ee^{\tau B}Q]\ee^{\tau A},
\end{displaymath}
 by the Duhamel formula
\begin{displaymath}
  R_Q(\tau)=\int_0^\tau \ee^{(\tau-\xi)(A+B)}\! \left\{[\ee^{\xi A}Q,B]\ee^{\xi B}+[A,\ee^{\xi B}Q]\ee^{\xi A}\right\}\!\D\xi=\O{\tau}.
\end{displaymath}
We are not done yet: Letting $F_Q(\xi)=[\ee^{\xi A}Q,B]$, we have $F_Q(0)=-[B,Q]$ and $F_Q'-AF_Q=[A,B]\ee^{\xi A}Q$, therefore
\begin{displaymath}
  F_Q(\xi)=-\ee^{\xi A}[B,Q]+\int_0^\xi \ee^{(\xi-\eta)A}[A,B]\ee^{\eta A}\D\eta \,Q.
\end{displaymath}
Likewise,
\begin{displaymath}
  [A,\ee^{\xi B}Q]=\ee^{\xi B}[A,Q]+\int_0^\xi \ee^{(\xi-\eta)B}[A,B]\ee^{\eta B}\D\eta\, Q
\end{displaymath}
and, after straightforward algebra,
\begin{Eqnarray*}
  R_{[A,B]}(\tau)&=&\int_0^\tau \ee^{(\tau-\xi)(A+B)}\!\left\{ -\ee^{\xi A}[B,[A,B]]\ee^{\xi B}+\ee^{\xi B}[A,[A,B]]\ee^{\xi A}\right\}\!\D\xi\\
  &&\mbox{}+\int_0^\tau \ee^{(\tau-\xi)(A+B)} \int_0^\xi \!\left\{ \ee^{(\xi-\eta)A}[A,B]\ee^{\eta A}\D\eta [A,B]\ee^{\xi B}\right.\\
  &&\hspace*{30pt}\left.\mbox{} +\ee^{(\xi-\eta)B}[A,B]\ee^{\eta B}\D\eta [A,B]\ee^{\xi A}\right\}\!\D\eta.
\end{Eqnarray*}

We can now rewrite \R{eq:2.3} in a form which is more amenable for an error bound. Since
\begin{Eqnarray*}
  &&\frac12 \int_0^t \tau \ee^{(t-\tau)(A+B)} \left\{ \ee^{\tau A}[A,B]\ee^{\tau B}-\ee^{\tau B}[A,B]\ee^{\tau A}\right\}\!\D\tau\\
  &=&\frac12 \int_0^t \tau \ee^{(t-\tau)(A+B)} R_{[A,B]}(\tau)\D\tau\\
  &=&\frac12 \int_0^t \tau \ee^{(t-\tau)(A+B)} \int_0^\tau \ee^{(\tau-\xi)(A+B)}\!\left\{ -\ee^{\xi A}[B,[A,B]]\ee^{\xi B}+\ee^{\xi B}[A,[A,B]]\ee^{\xi A}\right\}\!\D\xi\\
  &&\mbox{}+\frac12 \int_0^t \tau \ee^{(t-\tau)(A+B)}\int_0^\tau \ee^{(\tau-\xi)(A+B)} \int_0^\xi \!\left\{ \ee^{(\xi-\eta)A}[A,B]\ee^{\eta A}\D\eta [A,B]\ee^{\xi B}\right.\\
  &&\hspace*{30pt}\left.\mbox{} +\ee^{(\xi-\eta)B}[A,B]\ee^{\eta B}\D\eta [A,B]\ee^{\xi A}\right\}\!\D\eta\D\tau\\
  &=&\frac14 \int_0^t (t^2-\xi^2)\ee^{(t-\xi)(A+B)}\! \left\{ \ee^{\xi B}[A,[A,B]]\ee^{\xi A}-\ee^{\xi A}[B,[A,B]]\ee^{\xi B}\right\}\!\D\xi\\
  &&\mbox{}+\frac14 \int_0^t (t^2-\xi^2) \ee^{(t-\xi)(A+B)} \int_0^\xi \!\left\{ \ee^{(t-\eta)A}[A,B]\ee^{\eta A}\D\eta [A,B]\ee^{\xi B}\right.\\
  &&\hspace*{30pt}\left.\mbox{} +\ee^{(\xi-\eta)B}[A,B]\ee^{\eta B}\D\eta [A,B]\ee^{\xi A}\right\}\!\D\eta,
\end{Eqnarray*}
we deduce that 
\begin{Eqnarray}
   \nonumber
  &&E_{\PLT}(A,B;t)\\
  \nonumber
  &=&\frac14 \int_0^t (t^2-\xi^2)\ee^{(t-\xi)(A+B)}\! \left\{ \ee^{\xi B}[A,[A,B]]\ee^{\xi A}-\ee^{\xi A}[B,[A,B]]\ee^{\xi B}\right\}\!\D\xi\hspace*{20pt}\\
  \label{eq:2.5} 
  &&\mbox{}-\frac12 \int_0^t \ee^{(t-\tau)(A+B)} \int_0^\tau (\tau-\xi)\left\{ \ee^{(\tau-\xi)A}[A,[A,B]]\ee^{\xi A}\ee^{\tau B} \right.\\
  \nonumber
  &&\hspace*{20pt}\left.\mbox{}-\ee^{(\tau-\xi)B}[B,[A,B]]\ee^{\xi B}\ee^{\tau A}\right\}\!\D\xi\D\tau\\
  \nonumber
  &&\mbox{}+\frac14 \int_0^t (t^2-\xi^2) \ee^{(t-\xi)(A+B)} \int_0^\xi \!\left\{ \ee^{(\xi-\eta)A}[A,B]\ee^{\eta A}\D\eta [A,B]\ee^{\xi B}\right.\\
  \nonumber
  &&\hspace*{30pt}\left.\mbox{} +\ee^{(\xi-\eta)B}[A,B]\ee^{\eta B}\D\eta [A,B]\ee^{\xi A}\right\}\!\D\xi.
\end{Eqnarray}
The first two integrals are $\O{t^3}$ and the last one $\O{t^4}$.

\subsection{The Strang splitting}

Recalling from \R{eq:1.5} that
\begin{displaymath}
  Y_{\St}(A,B;t)=Y_{\LT}(A,B;\tfrac12 t)Y_{\LT}(B,A;\tfrac12 t),
\end{displaymath}
we have
\begin{Eqnarray*}
  &&E_{\St}(A,B;t)=Y_{\St}(A,B;t)-\ee^{t(A+B)}\\
  &=&\{E_{\LT}(A,B;\tfrac12 t)+\ee^{\frac12t(A+B)}\}\{E_{\LT}(B,A;\tfrac12 t)+\ee^{\frac12 t(A+B)}\} -\ee^{t(A+B)}\\
  &=&E_{\LT}(A,B;\tfrac12 t) \ee^{\frac12 t(A+B)}+\ee^{\frac12 t(A+B)} E_{\LT}(B,A;\tfrac12 t) \\
  &&\mbox{}+E_{\LT}(A,B;\tfrac12 t)\cdot E_{\LT}(B,A;\tfrac12 t).
\end{Eqnarray*}
While it is tempting (and possible!) to construct $E_{\St}$ from $E_{\LT}$ in this manner, it is simpler to proceed from first principles.

\begin{theorem}
  The error of Strang's splitting is
  \begin{Eqnarray}
  \label{eq:2.6}
  E_{\St}(A,B;t)&=&-\frac14 \int_0^t \ee^{(t-\tau)(A+B)} \ee^{\frac12\tau A}\int_0^\tau\! (\tau-\eta) \left\{ \ee^{-\frac12\eta A}[A,[A,B]]\ee^{\frac12\eta A}\right.\\
  \nonumber
  &&\hspace*{20pt}\left.\mbox{}+2\ee^{\eta B}[B,[A,B]]\ee^{-\eta B}\right\}\!\D\eta\,\ee^{\tau B}\ee^{\frac12 \tau A}\D\tau.
\end{Eqnarray}
\end{theorem}

\begin{proof}
Similarly to the proof of Theorem~\ref{th:LieTrotter},
\begin{eqnarray*}
  &&Y_{\St}(A,B;t)=\ee^{\frac12 tA}\ee^{tB}\ee^{\frac12 tA}\\
  \Rightarrow && \frac{\D Y_{\St}(A,B;t)}{\D t}-(A+B)Y_{\St}(A,B;t)=[\ee^{\frac12 tA},B]\ee^{tB}\ee^{\frac12 tA}+\frac12 \ee^{\frac12 tA}[\ee^{tB},A]\ee^{\frac12 tA}\\
  \Rightarrow && E_{\St}(A,B;t)=\int_0^t \ee^{(t-\tau)(A+B)}[\ee^{\frac12\tau A},B]\ee^{\tau B}\ee^{\frac12\tau A}\D\tau\\
  &&\hspace*{80pt}\mbox{}-\frac12 \int_0^t \ee^{(t-\tau)(A+B)}\ee^{\frac12\tau A} [A,\ee^{\tau B}]\ee^{\frac12\tau A}\D\tau.
\end{eqnarray*}
Let
\begin{displaymath}
  Z_1(\tau)=[\ee^{\frac12\tau A},B],\qquad Z_2(\tau)=[A,\ee^{\tau B}].
\end{displaymath}
Then $Z_1(0)=Z_2(0)=O$ and by the magic of Duhamel
\begin{eqnarray*}
  Z_1-\frac12 AZ_1=\frac12[A,B]\ee^{\frac12\tau A}&\Rightarrow&Z_1(\tau)=\frac12 \int_0^\tau \ee^{\frac12(\tau-\xi)A}[A,B]\ee^{\frac12\xi A}\D\xi,\\
  Z_2'-Z_2B=\ee^{\tau B}[A,B]&\Rightarrow&Z_2(\tau)=\int_0^\tau \ee^{\xi B}[A,B]\ee^{(\tau-\xi)B}\D\xi.
\end{eqnarray*}
Therefore
\begin{Eqnarray*}
  &&E_{\St}(A,B;t)\\
  &=&\frac12 \int_0^t \ee^{(t-\tau)(A+B)} \ee^{\frac12\tau A} \int_0^\tau \ee^{-\frac12\xi A}[A,B]\ee^{\frac12\xi A}\D\xi \ee^{\tau B}\ee^{\frac12\tau A}\D\tau\\
  &&\mbox{}-\frac12 \int_0^t \ee^{(t-\tau)(A+B)}\ee^{\frac12\tau A}\int_0^\tau \ee^{\xi B}[A,B]\ee^{-\xi B}\D\xi \ee^{\tau B}\ee^{\frac12\tau A}\D\tau\\
  &=&\frac12\int_0^t \ee^{(t-\tau)(A+B)}\ee^{\frac12\tau A}\int_0^\tau \!\left\{ \ee^{-\frac12\xi A}[A,B]\ee^{\frac12\xi A}-\ee^{\xi B}[A,B]\ee^{-\xi B}\right\} \!\ee^{\frac12\tau A}\D\tau.
\end{Eqnarray*}
Recalling the operator $R_Q$ from the previous section, we let
\begin{displaymath}
  \tilde{R}_Q(\xi)=\ee^{-\frac12 \xi A}Q\ee^{\frac12\xi A}-\ee^{\xi B}Q\ee^{-\xi B}.
\end{displaymath}
Since $\tilde{R}_Q(0)=O$, we have
\begin{displaymath}
  \tilde{R}_Q(\xi)=\int_0^\xi \tilde{R}_Q'(\eta)\D\eta=-\frac12 \int_0^\xi\! \left\{ \ee^{-\frac12\eta A}[A,Q]\ee^{\frac12\eta A}+2\ee^{\eta B}[B,Q]\ee^{-\eta B}\right\}\!\D\eta.
\end{displaymath}
The theorem follows because
\begin{displaymath}
  E_{\St}(A,B;t)=\frac12\int_0^t \ee^{(t-\tau)(A+B)}\ee^{\frac12\tau A} \int_0^\tau \tilde{R}_{[A,B]}(\xi)\ee^{\frac12\tau A}\D\tau.
\end{displaymath}
We obtain 
\begin{Eqnarray*}
  E_{\St}(A,B;t)&=&-\frac14 \int_0^t \ee^{(t-\tau)(A+B)} \ee^{\frac12\tau A} \int_0^\tau  \int_0^\xi\! \left\{ \ee^{-\frac12\eta A}[A,[A,B]]\ee^{\frac12\eta A}\right.\\
  &&\hspace*{20pt}\left.\mbox{}+2\ee^{\eta B}[B,[A,B]]\ee^{-\eta B}\right\}\!\D\eta\D\xi\,\ee^{\tau B}\ee^{\frac12 \tau A}\D\tau
\end{Eqnarray*}
and, exchanging the two inner integrals, derive \R{eq:2.6}.
\end{proof}

Note that \R{eq:2.6} implies that $E_{\St}(A,B;t)=\O{t^3}$.

\section{Error bounds}

A standard means of bounding exponentials of matrices is by using a {\em logarithmic norm\/} \cite{soderlind2006tln}. Thus, let $A$ be a complex-valued square matrix, $A:\mathcal{Y}\rightarrow\mathcal{Y}$, where  $\mathcal{Y}$ is a normed space. The limit
\begin{equation}
  \label{eq:3.1}
  \mu[A]=\lim_{\varepsilon\downarrow0} \frac{\|I+\varepsilon A\|-1}{\varepsilon},
\end{equation}
where $\|\,\cdot\,\|$ is the norm of $\mathcal{Y}$, is the logarithmic norm of $A$. This concept is important because
\begin{equation}
  \label{eq:3.2}
  \|\ee^{tA}\|\leq \ee^{t\mu[A]},\qquad t\geq0,
\end{equation}
and $\mu[A]$ is the least number with this property. In particular, in the case of the standard Euclidean norm, $\mu[A]$ is the largest eigenvalue of the symmetric matrix $\frac12(A+A^*)$. Note that $\mu[A]$ might be negative: a logarithmic norm is not a norm!

The definition \R{eq:3.1} and the inequality \R{eq:3.2} survive once $\mathcal{Y}$ is replaced by a Hilbert space and $A$ is a linear operator \cite{soderlind2006tln}, while in the case of Euclidean norm $\mu[A]=\sup \sigma(\frac12(A+A^*))$, where $\sigma$ is the spectrum -- it is bounded once $A$ generates a strongly continuous semigroup and $\ee^{tA}$ is well posed. For many cases of interest $\mu[A]=0$ but we do not need to impose this condition. 

We let
\begin{displaymath}
  \omega=\mu[A]+\mu[B]-\mu[A+B].
\end{displaymath}
The error bounds are expressed in terms of $\omega$, rather than individual logarithmic norms, and this renders them substantially simpler.

\subsection{Lie--Trotter}

It follows at once from \R{eq:2.1} that
\begin{equation}
  \label{eq:3.3}
  \|E_{\LT}(A,B;t)\|\leq \ee^{t\mu[A+B]} \frac{1-(1-\omega t)\ee^{\omega t}}{\omega^2}\|[A,B]\|.
\end{equation}
The most ubiquitous case is $\mu[A]=\mu[B]=\mu[A+B]=0$ -- in that case letting $\omega\rightarrow0$ in \R{eq:3.3} results in $\|E_{\LT}(A,B;T)\|\leq \frac12 t^2 \|[A,B]\|$. This upper bound is stronger than the expression \R{eq:1.6}.

\subsection{Palindromic Lie--Trotter}

Palindromic Lie--Trotter splitting leads to a more complicated error bound,
\begin{Eqnarray*}
  &&\|E_{\PLT}(A,B;t)\|\\
  &\leq&\ee^{t\mu[A+B]}(1+\Frac12 \omega t) \frac{(1+\frac12\omega t)\ee^{-\omega t}-(1-\frac12 \omega t)}{\omega^3} (\|[A,[A,B]]\|+\|[B,[A,B]]\|)\\
  &&\mbox{}+\frac{3}{\omega^4}\left[\left(1+\omega t+\frac13\omega^2t^2\right)\ee^{-\omega t}-\left(1-\frac16 \omega^2 t^2\right)\right] \|[A,B]\|^2
\end{Eqnarray*}
and the error bound becomes
\begin{displaymath}
  \frac{t^3}{12} (\|[A,[A,B]]\|+\|[B,[A,B]]\|)+\frac{t^4}{8}  \|[A,B]\|^2
\end{displaymath}
once $\omega\rightarrow0$ -- compare with \R{eq:1.7}.

\subsection{Strang}

The Strang splitting lends itself to considerably neater error bound. Starting from \R{eq:2.6}, we have
\begin{Eqnarray*}
  &&\|E_{\St}(A,B;t)\|\\
  \nonumber
  &\leq&\frac14 \ee^{t\mu[A+B]} \frac{(1-\omega t+\frac12\omega^2t^2)\ee^{\omega t}-1}{\omega^3}(\|[A,[A,B]]\|+2\|[B,[A,B]]\|).
\end{Eqnarray*}
Once $\mu[A],\mu[B],\mu[A+B]=0$, thus $\omega\rightarrow0$, we recover
\begin{displaymath}
  \|E_{\St}(A,B;t)\|\leq \frac{t^3}{24} (\|[A,[A,B]]\|+2\|[B,[A,B]]\|)
\end{displaymath}
-- the leading error term in \R{eq:1.8} becomes an upper bound! 

\section{Conclusion}

Exponential splittings are an important strategy in computing `difficult' time-dependent partial differential equations by converting them into problems which are considerably easier to evaluate. While the procedure is elementary and intuitive, its analysis is difficult, not least because of the presence of unbounded operators. Typically, this calls for advanced methods of functional analysis and semigroup theory \cite{hansen09esu}. In this paper we demonstrate that this analysis can be accomplished by elementary means, accessible to advanced mathematics undergraduates. 

\section*{Acknowledgements}

The work of KK in this project has been supported by The National Center for Science (NCN), based on Grant No.\ 2019/34/E/ST1/00390.

\end{document}